\documentclass{amsart}

\usepackage{amsmath,amsfonts,amssymb,amscd,verbatim,latexsym}
\usepackage{hyperref}
\usepackage{enumerate}
\usepackage{bm}
\usepackage{color}
\usepackage{makecell}

\numberwithin{equation}{section}
\theoremstyle{plain}
\newtheorem{thm}{Theorem} 
\newtheorem{cor}[thm]{Corollary} 
\newtheorem{theorem}[equation]{Theorem} 

\newtheorem{corollary}[equation]{Corollary} 
\newtheorem{lemma}[equation]{Lemma}
\newtheorem{proposition}[equation]{Proposition}

\theoremstyle{definition}

\newtheorem{example}[equation]{Example}
 
\newtheorem{question}[equation]{Question}
\newtheorem{remark}[equation]{Remark}

\DeclareMathOperator\Aut{Aut}

\DeclareMathOperator\diag{diag}
\DeclareMathOperator\End{End}
\DeclareMathOperator\gr{gr}
\DeclareMathOperator\GKdim{GKdim}
\DeclareMathOperator\gldim{gl. dim}

\DeclareMathOperator\hdet{hdet}
\DeclareMathOperator\id{id}

\DeclareMathOperator\Tr{Tr}
\DeclareMathOperator\VdM{Vdm}
\DeclareMathOperator\CM{CM}

\newcommand\NN{\mathbb N}
\newcommand\PP{\mathbb P}

\newcommand\ZZ{\mathbb Z}

\newcommand\cS{\mathcal S}

\newcommand{\Ext}{\ensuremath{\operatorname{\underline{Ext}}}}
\newcommand{\tails}{\ensuremath{\operatorname{tails}}}
\newcommand{\tors}{\ensuremath{\operatorname{tors}}}
\newcommand{\grmod}{\ensuremath{\operatorname{grmod}}}

\renewcommand{\int}{\smallint}
\newcommand\inv{^{-1}}

\newcommand\iso{\cong}
\newcommand\kk{\Bbbk}
\newcommand\tensor{\otimes}

\newcommand{\p}{{\sf{p}}}
\renewcommand{\r}{{\sf{r}}}

\newcommand{\GL}{\ensuremath{\operatorname{GL}}}

\renewcommand{\to}{\ensuremath{\longrightarrow}}

\newcommand\cF{\mathcal F}
\newcommand\hf{\hat{f}}

\newcommand\grp[1]{\left\langle #1 \right\rangle}

\begin{document}

\title[Auslander's Theorem for noncommutative algebras]{Auslander's Theorem for permutation actions on noncommutative algebras}

\author[Gaddis]{Jason Gaddis}
\address{Miami University, Department of Mathematics, 301 S. Patterson Ave., Oxford, Ohio 45056} 
\email{gaddisj@miamioh.edu}

%\author[Gaddis]{Jason Gaddis}
%\address{Wake Forest University, Department of Mathematics and Statistics, P. O. Box 7388, Winston-Salem, NC 27109} 
%\email{gaddisjd@wfu.edu, kirkman@wfu.edu, moorewf@wfu.edu, wonrj@wfu.edu}

\author[Kirkman]{Ellen Kirkman}
\address{Wake Forest University, Department of Mathematics and Statistics, P. O. Box 7388, Winston-Salem, North Carolina 27109} 
\email{kirkman@wfu.edu}

\author[Moore]{W. Frank Moore}
\address{Wake Forest University, Department of Mathematics and Statistics, P. O. Box 7388, Winston-Salem, North Carolina 27109}
\email{moorewf@wfu.edu}

\author[Won]{Robert Won}
\address{University of Washington, Department of Mathematics, Box 354350, Seattle, Washington 98195}
\email{wonrj@wfu.edu}

%\thanks{}
%\date{\today}

\commby{Jerzy Weyman}

\begin{abstract}
%Let $A$ be a $\kk$-algebra with $\kk$ a field, $G$ a finite group acting
%linearly on $A$ by automorphisms, and $A^G$ the subring of invariants.
When $A = \kk[x_1, \ldots, x_n]$ and $G$ is a small subgroup of $\GL_n(\kk)$,
Auslander's Theorem says that the skew group algebra $A \# G$ is isomorphic to
$\End_{A^G}(A)$ as graded algebras.
%Auslander's Theorem plays an important role in the classical McKay correspondence. 
We prove a generalization of Auslander's Theorem for permutation actions on 
$(-1)$-skew polynomial rings, $(-1)$-quantum Weyl algebras, 
three-dimensional Sklyanin algebras, and a certain homogeneous down-up algebra. 
We also show that certain fixed rings $A^G$ are graded isolated singularities
in the sense of Ueyama.
\end{abstract}

\maketitle
\section{Introduction}

Throughout, $\kk$ is 
an algebraically closed field of characteristic zero. All algebras are $\kk$-algebras and
all tensor products are over $\kk$.
For an algebra $A$ and $G$ a group of algebra 
automorphisms of $A$,
the set of all elements of $A$ that are invariant under the 
action of $G$ (denoted by $A^G$) forms a subalgebra of $A$,  
called the ring of invariants of $G$ on $A$.  Many notions in 
commutative algebra have their roots in the study of properties of 
rings of invariants of finite group actions on 
$\kk[x_1,\dots,x_n]$.  

The homological properties of $A^G$ are often of particular interest.  A
classical result of Shephard-Todd and Chevalley \cite{Chev,SheTod} states that
if $A = \kk[x_1,\dots,x_n]$ and $G$ is a finite subgroup of $\GL_n(\kk)$
acting as graded automorphisms, then $A^G$ has finite global dimension if and
only if $G$ is generated by reflections; an element $g \in \GL_n(\kk)$ is said to be a 
{\sf reflection} if $g$ fixes a codimension-one subspace of
$V= \bigoplus_{i=1}^n \kk x_i$.  
%The group $G$ is said to be {\sf small} if it contains no reflections.

When $A$ is not the commutative polynomial ring the above definition of
reflection is not appropriate, but there is a suitable
generalization due to Kuzmanovich, Zhang and the second author \cite[Definition 2.2]{KKZ1}.
For a graded algebra $A$ of Gelfand-Kirillov (GK) dimension $n$, a graded automorphism $g$ is a reflection if its trace series is of the form
\[ \Tr_A(g,t) := \sum_{n \geq 0} \Tr(g|_{A_n})t^n = 
\frac{1}{(1-t)^{n-1}q(t)} \text{ for } q(1) \neq 0,\]
where $A_n$ denotes the $n^\text{th}$ graded component of $A$ and $\Tr(g|_{A_n})$ denotes the usual
trace of the action of $g$ on $A_n$.
When $A = \kk[x_1,\dots,x_n]$, $g$ is a reflection in 
this sense if and only if it is a reflection in the classical 
sense. In \cite[Theorem 5.5]{KKZ4}, the same authors prove that if
$A = \kk_{p_{ij}}[x_1,\dots,x_n]$ (a skew polynomial ring) and $G$ is a
finite group of graded automorphisms of $A$, then $A^G$ has
finite global dimension if and only if $G$ is generated by
reflections in this more general sense. 

Another classical theorem from invariant theory in which
reflections play an important role is a theorem of Auslander.  The {\sf skew group algebra} $A\# G$ is $A \tensor \kk G$ as a 
vector space over $\kk$, with multiplication defined by
$(a \tensor g)(b \tensor h) = a (g.b) \tensor gh$ for all 
$a,b \in A$, $g,h \in G$.
A natural algebra homomorphism (due to Auslander) from
$A \# G$ to $\End_{A^G}(A)$ (where we view $A$ as a \emph{right}
$A^G$-module) is given by
\begin{eqnarray*}
\gamma_{A,G} : A \# G & \to & \End_{A^G}(A) \\
       a \# g & \mapsto & 
       \left(\begin{matrix}A & \to & A \\ b & \mapsto & ag(b)\end{matrix}\right).
\end{eqnarray*}

\begin{thm} {\rm (Auslander \cite{A})}
If $A = \kk[x_1,\dots,x_n]$, $|G|$ is invertible in $\kk$, and $G$ does
not contain any nontrivial reflections (i.e., $G$ is a ``small group"), then
$\gamma_{A,G}$ is an isomorphism.
\end{thm}

We call the map
$\gamma_{A,G}$ the {\sf Auslander map} of the $G$-action on $A$.
In general, this map is neither injective nor surjective.  The category of (left) modules over $A \# G$ is equivalent to the category of
(left) $A$-modules with compatible left $G$-actions and their homomorphisms. Auslander's Theorem identifies an $A^G$-module (namely $A$) whose endomorphism ring
has finite global dimension.  

%The natural embedding $A \rightarrow A \# G$ given by 
%$a \mapsto a \tensor e$,
%where $e$ is the identity of $G$, allows us to identify $A$ with its image in
%$A \# G$.

%We say Auslander's Theorem holds for an algebra $A$
%and a group $G$ acting on $A$ linearly if there exists
%an isomorphism $A\# G \rightarrow \End_{A^G}(A)$.
%Classically, this holds for $G$ a small finite subgroup of $\GL_n(\kk)$ acting 
%linearly on $A=\kk[x_1,\hdots,x_n]$ \cite{A} and \cite[Theorem 5.15]{LW}.

%\begin{thm}[Auslander's Theorem]
%\cite[Proposition 3.4]{A} and \cite[Theorem 5.15]{LW}.
%Let $G$ be a small finite subgroup of $\GL_n(\kk)$ acting linearly
%on $A=\kk[x_1,\hdots,x_n]$. Then the skew group ring $A\# G$
%is isomorphic to $\End_{A^G}(A)$ as graded algebras.
%\end{thm}

Natural generalizations of Auslander's result include replacing $A$ with a noncommutative ring and
replacing the action of the group $G$ with the action
of a Hopf algebra. Significant progress has been made in both directions,
see \cite{BHZ2,BHZ1,CKWZ2,CKWZ1,EW,M,MU}.
%Could cite HVOZ above.
Our interest is in the former. It is conjectured that the Auslander map 
is an isomorphism for a noetherian AS regular domain $A$ and a group
$G$ that does not contain a reflection.
%Significant progress on this conjecture has been made in \cite{BHZ2,BHZ1}.

Bao, He, and Zhang \cite{BHZ2,BHZ1} (using some ideas that appeared
in an earlier paper of Buchweitz \cite{Buch}) prove that the
Auslander map is related to an invariant of the
$G$-action on $A$ known as the pertinency.  
Let $A$ be an affine algebra generated in degree 1 and $G$ a finite 
subgroup of $\GL_n(\kk)$ acting on $A_1$.
The {\sf pertinency} of the $G$-action on $A$ 
\cite[Definition 0.1]{BHZ2} is defined to be
\[ \p(A,G) = \GKdim A - \GKdim (A\# G)/(f_G)\]
where $(f_G)$ is the two sided ideal of $A\#G$ generated by
$f_G = \sum_{g \in G} 1\# g$, and GKdim is the Gelfand-Kirillov (GK) dimension of the algebra. 
In what follows, we often drop the subscript on $f_G$ when the group
$G$ is clear from the context.  

%\begin{remark}
%\end{remark}
%Bao, He, and Zhang prove that whether Auslander's map is an isomorphism depends 
%only on the pertinency of the $G$-action on $A$.

A connected $\NN$-graded algebra $A$ is {\sf{Artin-Schelter (AS) regular}}
provided it has finite GK dimension, finite global dimension $d$, and 
satisfies $\Ext_A^i(_{A}\kk,A)=\delta_{id} \kk(\ell)$ for some $\ell \in \ZZ$. An algebra
$A$ is {\sf{Cohen-Macaulay (CM)}} if for every graded $A$ module
$j(M)+\GKdim(M)=\GKdim(A)$ where $j(M)=\min\{ i : \Ext_A^i(M,A)\neq 0\}$.

\begin{thm}[{\cite[Theorem 0.3]{BHZ1}}] \label{thm.bhz}
Let $A$ be a noetherian, connected graded, AS regular, CM
$\kk$-algebra of GK dimension at least $2$. Let $G$ be a group acting linearly 
on $A$. Then the Auslander map is an isomorphism if and only if $\p(A,G) \geq 2$.
\end{thm}

%Our main technique for computing the pertinency of a group action is based on
%an argument of Brown and 
%Lorenz \cite[Lemma 2.2]{BL}.  They prove that for a commutative algebra $A$,
%one can produce an element in $(f_G) \cap A$ by taking $f_G$-commutators.
%Two problems with directly applying their result to our situation are that
%commutativity is used heavily, and that the element produced has degree $n!-1$, which is 
%much higher than the degree where nonzero elements of $(f_G)$ first 
%appear in $A$. We modify their method by carefully choosing central
%(or normal) elements that produce elements that 
%bound the pertinency below and prove Theorem \ref{thm.main}.

In this paper we compute lower bounds on the pertinency of some group actions on several families of noetherian AS regular domains including
\begin{itemize}
\item the {\sf $(-1)$-skew polynomial algebras} $V_n$, 
generated by $x_1,\hdots,x_n$ subject to the relations $x_ix_j+x_jx_i=0$ 
for $i \neq j$;

\item the {\sf $(-1)$-quantum Weyl algebras} $W_n$, generated 
by $x_1,\hdots,x_n$ with relations $x_ix_j+x_jx_i=1$ for $i \neq j$;
taking the standard filtration, $\gr(W_n)=V_n$;

\item the {\sf three-dimensional Sklyanin algebras} $S(a,b,c)$,
generated by $x_1,x_2,x_3$ with parameters $(a:b:c) \in \PP^2$ and relations 
\[ax_1x_2 + bx_2x_1 + cx_3^2 = ax_2x_3 + bx_3x_2 + cx_1^2 = ax_3x_1 + bx_1x_3 + cx_2^2 = 0;\]

\item the {\sf graded noetherian down-up algebras} $A(\alpha,\beta)$ with parameters 
$\alpha,\beta \in \kk$ with $\beta \neq 0$ generated by $x$ and $y$ with relations
\[	x^2y = \alpha xyx + \beta yx^2 \mbox{ and } xy^2 = \alpha yxy + \beta y^2x.\]
\end{itemize}
%$V_n$ is a noetherian AS regular domain, and
The symmetric group $\cS_n$ acts on $V_n$ by permutations; that is, if $\sigma \in \cS_n$ then $\sigma(x_i)=x_{\sigma(i)}$,
extending linearly and multiplicatively.
Throughout we assume that any subgroup of $\cS_n$ acts on $V_n$
in this way unless otherwise stated.
For this action on $V_n$, $\cS_n$ does not contain any reflections \cite[Lemma 1.7(4)]{KKZ2}
and our result (Theorem \ref{thm.main}) proves that the Auslander map is an isomorphism 
for this action on $V_n$.
This result is quite different than for commutative polynomial rings, 
where the Auslander map is not an isomorphism for subgroups of $\cS_n$ that contain a transposition. 
%Bao, He, and Zhang proved that the Auslander map is an isomorphism for a subgroup $\cS_n$
%generated by an $n$-cycle \cite[Theorem 5.7]{BHZ1}.
In Section \ref{sec.comp} we compute the pertinency of
most subgroups of $\cS_3$ and $\cS_4$ acting on $V_3$ and $V_4$, respectively. 

By \cite[Theorem 0.2]{Z2}, \cite[Corollary 6.2]{L}, and \cite{KMP}, 
$V_n$, $S(a,b,c)$, and $A(\alpha,\beta)$ are AS regular and CM and thus satisfy the hypotheses of Theorem \ref{thm.bhz}. In the following theorem,
%For certain groups $G$ acting on these algebras, 
for each $A$ and $G$ given, we verify that $\p(A,G) \geq 2$, and hence
    prove by Theorem \ref{thm.bhz} that the Auslander map is an isomorphism.
\begin{thm} \label{thm.allaus} 
The Auslander map is an isomorphism for the following groups acting on the following algebras:
\begin{enumerate} 
\item any subgroup of $\cS_n$ acting on $V_n$ (Theorem \ref{thm.main}),
\item any subgroup of $\cS_n$ acting on $W_n$ (Theorem \ref{thm.main}),
\item any subgroup of $\cS_3$ acting on $S(1,1,-1)$ (Theorem \ref{thm.main}),
\item any subgroup of weighted permutations acting on $A(-2,-1)$ (Theorem \ref{thm.special}),
\item $\grp{(1~2~3)}$ acting on a generic Sklyanin algebra $S(a,b,c)$ (Theorem \ref{thm.sklyanin}), and
\item any subgroup of $\grp{-I_4, (1~3)(2~4)}$ acting on $V_4$ (Theorem \ref{thm.weighted}).
%\item and the group $\grp{-I_n, \left[ \begin{smallmatrix}0 & 0 & 1 & 0 \\ 0 & 0 & 0 & 1 \\ 1 & 0 & 0 & 0 \\ 0& 1 & 0 & 0 \end{smallmatrix}\right]}$ acting on $V_4$ (Theorem \ref{thm.weighted}).
\end{enumerate}
\end{thm}

To bound $\p(A,G)$, we make frequent use of the following theorem.

%\begin{thm}[{\cite[Lemma 5.2]{BHZ1}}]
%Let $T$ be a subalgebra of an algebra $R$ such that $R_T$ and ${}_T R$
%are finitely generated.
%Let $R'$ be the image of the map $R \rightarrow B \rightarrow B/I$
%and $T'$ be the image of $T$ in $R'$.
%Then $\GKdim T' = \GKdim R' = \GKdim B/I$.
%\end{thm}

\begin{thm}[Theorem {\ref{thm.subgrp}}]
Let $G$ be a group acting on $A$ and let $H \leq G$ be a subgroup. Then $\p(A,G) \leq \p(A,H)$.
\end{thm}

When $A^{\grp{g}}$ has finite global dimension,  $\p(A,\grp{g}) < 2$, and we obtain the following corollary.
\begin{cor}[Corollary {\ref{cor.refl}}]
Let $A$ be a noetherian connected graded $\kk$-algebra and let $G$ be a finite subgroup of $\Aut_{\gr}(A)$. 
Suppose that $g \in G$ is a reflection. If $A$ and $A^{\langle g \rangle}$ have finite global
dimension, then the Auslander map $\gamma_{A,G}$ is not an isomorphism.
\end{cor}
\noindent We remark that $A^{\grp{g}}$ has finite global dimension when $g$ is a reflection acting on an
AS regular noetherian domain with Hilbert series $(1-t)^{-n}$ \cite[Theorem 5.3]{KKZ1}, and hence this corollary applies to groups that contain a reflection acting on $V_n$.

Auslander's Theorem is a component of the classical McKay correspondence. 
In \cite{CKWZ2, CKWZ1}, Chan, Walton, Zhang, and the second-named author study a 
quantum version of the McKay correspondence.
If $A$ is AS regular and $g \in \Aut_{\gr}(A)$,
then {\sf{the homological determinant of $g$}}, denoted $\hdet(g)$, may be computed
using the following formula \cite[Lemma 2.6]{JZ}:
\[ \Tr_A(g,t) = (-1)^n \hdet(g)^{-1} t^{-\ell} + (\text{lower order terms}).\]
With the exception of certain weighted actions on the down-up algebra,
all actions of groups $G$ on $A$ in Theorem \ref{thm.allaus} have {\sf{trivial homological determinant}},
that is, $\hdet(g)=1$ for all $g \in G$.
%and we obtain the following correspondences as a corollary. 
Hence, as a consequence of our work, the following part of the McKay correspondence applies to the algebras of Theorem \ref{thm.allaus}.
Following the terminology in \cite{CKWZ2}, call an $A$-module $M$ {\sf initial} if $M$ is graded, 
generated in degree $0$ and $M_{< 0} = 0$.
%When the Auslander map is an isomorphism, the authors above prove several correspondences. Therefore, for the algebras and groups in Theorem \ref{thm.allaus}, 
\begin{thm}[{\cite[Theorem A]{CKWZ2}}] \label{thm.mckay} 
Let $A$ be a noetherian AS regular algebra and $G$ a finite group acting
on $A$ as graded automorphisms with trivial homological determinant. 
If $A \# G \cong \End_{A^G} (A)$, then there are bijective correspondences 
between the isomorphism classes of
\begin{itemize}
\item simple left $G$-modules,
\item indecomposable direct summands of $A$ as left $A^G$-modules, and
\item indecomposable, finitely generated, projective, initial, left
  %$\End_{A^G}(A) = 
  $A \# G$-modules.
\end{itemize}
\end{thm}

We conclude in Section \ref{sec.sing} by providing examples of graded isolated singularities in 
the sense of Ueyama \cite{U}. For a graded algebra $A$, let $\grmod A$ denote the category of 
finitely-generated graded right $A$-modules. For a module $M \in \grmod A$, $x \in M$ is called 
torsion if there exists a positive integer $n$ such that $x A_{\geq n} = 0$. The module $M$ is 
called a torsion module if every element of $M$ is torsion. 
Let $\tors A$ denote the full subcategory of $\grmod A$ consisting of torsion modules.
We can then define the quotient 
category $\tails A = \grmod A / \tors A$. Following \cite{U}, we say that $A^G$ is a {\sf graded isolated singularity} if $\gldim(\tails A^G) < \infty$. Mori and Ueyama prove that if the Auslander map is an isomorphism, then $A^G$ is a graded isolated singularity if and only if $A\#G/(f_G)$ is finite-dimensional \cite[Theorem 3.10]{MU}. For several algebras $A$ and groups $G$, we are able to show that $A\#G/(f_G)$ is finite-dimensional.  These examples are of particular interest, since for a graded isolated singularity $A^G$, the category of graded CM $A^G$-modules has several nice properties (we refer the reader to \cite{U2} for undefined terminology).  %Call an $A$-module {\sf graded maximal Cohen-Macaulay} if $\Ext_A^i(M,A) = 0$ for all $i \neq 0$. Let $\CM^{\gr}(A)$ denote the full subcategory of $\grmod A$ consisting of graded maximal Cohen-Macaulay modules. 
\begin{thm}[{\cite[Theorem 3.10 and Example 3.13]{U2}}] \label{thm.isol} Let $A$ be a noetherian AS regular algebra of dimension $d \geq 2$, and let $G$ be a group acting linearly on $A$ with trivial homological determinant. If $A^G$ is a graded isolated singularity, then
\begin{itemize}
\item $A^G$ is an AS Gorenstein algebra of dimension $d \geq 2$,
\item $A \in \CM^{\gr}\left(A^G\right)$ is a $(d-1)$-cluster tilting module, and
\item for $M \in \CM^{\gr}\left(A^G\right)$ the $\kk$-vector spaces $\Ext_{A^G}^1(A, M)$ and $\Ext_{A^G}^1(M,A)$ are finite-dimensional.
\end{itemize}
\end{thm}

We show that $A^G$ is a graded isolated singularity in the following cases.

\begin{thm} For the following groups $G$ acting on the following $\kk$-algebras $A$, $A^G$ is a graded isolated singularity:
\begin{enumerate} 
\item $\grp{(1~2)(3~4),(1~3)(2~4)}$ acting on $V_4$ (Proposition \ref{prop.rob}),
\item $\grp{(1~2)(3~4) \cdots(2n-1~2n)}$ acting on $V_{2n}$ (Proposition \ref{prop.frank}),
\item $\grp{(1~2~\cdots~2^n)}$ acting on $V_{2^n}$ (\cite[Theorem 5.7]{BHZ1}),
\item $\grp{(1~2~3)}$ acting on a generic Sklyanin algebra $S(a,b,c)$ (Theorem \ref{thm.sklyanin}), and
\item any subgroup of $\grp{-I_4, (1~3)(2~4)}$ acting on $V_4$ (Theorem \ref{thm.weighted}).
\end{enumerate}
\end{thm}

\section{\texorpdfstring{$(-1)$}{(-1)}-skew polynomials and related algebras}
\label{sec.elts}

In this section, we prove that if $G$ is a finite subgroup of $\cS_n$ acting as permutations on the $(-1)$-skew polynomial ring $V_n$, then the pertinency $\p(V_n,G)$ is bounded below by $2$, thus proving that, in this case, the Auslander map is an isomorphism by Theorem \ref{thm.bhz}.
We also consider related algebras including 
the Sklyanin algebra $S(1,1,-1)$ and 
the $(-1)$-quantum Weyl algebras $W_n$. The idea for the computation that follows was inspired by
an argument of Brown and 
Lorenz \cite[Lemma 2.2]{BL} for a commutative algebra $A$.
%one can produce an element in $(f_G) \cap A$ by taking $f_G$-commutators.
%Two problems with directly applying their result to our situation are that
%commutativity is used heavily, and that the element produced has degree $n!-1$, which is 
%much higher than the degree where nonzero elements of $(f_G)$ first 
%appear in $A$. We modify their method by carefully choosing central
%(or normal) elements that produce elements that 
%bound the pertinency below and prove Theorem \ref{thm.main}.

Let $A$ be a quotient of the free algebra $\kk \langle x_1, \ldots, x_n \rangle$ 
such that $\cS_n$ acts on $A$. 
Denote by $C(A)$ the center of $A$. 
We assume $x_\ell^2 \in C(A)$ for $1 \leq \ell \leq n$, which
clearly holds for $V_n$ and $W_n$, and
is well-known for $S(1,1,-1)$ (see \cite[Section 8.2]{WWY}).

%\begin{lemma}
%\label{lem.skcenter} The elements $x_1^2,x_2^2,x_3^2$ are central in $S(1,1,-1)$.
%\end{lemma}
%
%\begin{proof}
%By symmetry, if suffices to verify the claim for $x_1^2$.
%Multiplying the relation $x_1x_3+x_3x_1=x_2^2$ on either side by $x_1$ and
%subtracting the resulting relations gives 
%$x_1^2x_3-x_3x_1^2=x_1x_2^2-x_2^2x_1$.
%Cyclic permutations now give the additional relations
%\begin{align*}
%x_1^2x_3-x_3x_1^2 &= x_1x_2^2-x_2^2x_1 \\
%x_2^2x_1-x_1x_2^2 &= x_2x_3^2-x_3^2x_2 \\
%x_3^2x_2-x_2x_3^2 &= x_3x_1^2-x_1^2x_3.
%\end{align*}
%Thus,
%$x_1^2x_3-x_3x_1^2
%	= x_1x_2^2-x_2^2x_1
%    = -(x_2x_3^2-x_3^2x_2)
 %   = x_3x_1^2-x_1^2x_3$ and so
%$x_1^2x_3-x_3x_1^2=0$.
%Similarly, $x_1^2x_2-x_2x_1^2=0$ and so $x_1^2$ is central.
%\end{proof}

Let $G$ be a subgroup of $\cS_n$ and set $f = \sum_{\sigma \in G} 1 \# \sigma$. 
For $i < j$, we define
\[f_{i,j} = (x_i - x_j) \prod_{\substack{ (a,b) \neq (i,j)	\\ a < b}} (x_a^2 - x_b^2).\]

\begin{lemma} 
\label{lem.elts}
The elements $f_{i,j} \in (f) \cap A$ for $1 \leq i < j \leq n$.
\end{lemma}
\begin{proof} 
Let $U = \{ (a,b) \mid 1 \leq a < b \leq n, (a,b) \neq (i,j)\}$ and let $m = |U|$. Enumerate the elements of $U$:
$\{ (a_1, b_1),\ldots,(a_m, b_m)\}$.
Set $p_0 = f$. For $1 < k \leq m$, we define $p_k$ recursively. 
Let $p_{k} = x_{a_k}^2 p_{k-1} - p_{k-1}x_{b_k}^2$.
Since $p_0 = f$, it is clear that $p_k \in I$ for all $k$. 
Because the $x_\ell^2$ are in $C(A)$ by hypothesis, a direct computation shows
\[p_m = \sum_{\sigma \in G}\left( (x_{a_1}^2-x_{\sigma(b_1)}^2) \cdots (x_{a_m}^2-x_{\tau(b_m)}^2) \# \sigma \right).\]
Further, since $U$ was taken over all $a < b$ except $(a,b) = (i,j)$, if $\sigma \neq e$ and $\sigma \neq (i~j)$, 
then there exists some $a_k, b_k$ such that $\sigma(b_k) = a_k$. 
Hence, $p_m$ vanishes on all components except for the identity 
and possibly the transposition $(i~j)$, if $(i~j) \in G$. 
We now observe that $x_i p_m - p_m x_j = f_{i,j} \in (f)$, completing the proof.
\end{proof}

We define the Vandermonde determinant on the elements $y_1,\dots,y_n$ in $T = \kk[y_1,\dots,y_n]$
by $\VdM(T) = \prod_{i < j} (y_i - y_j)$.  If $A=V_n$ or $A=S(1,1,-1)$ (with $n=3$), we set $y_i = x_i^2 \in C(A)$ and $T \subset C(A)$.
By Lemma \ref{lem.elts}, $f_{1,2} \in (f) \cap A$.
It follows that $\VdM(T) = x_1f_{1,2}+f_{1,2}x_2 \in (f) \cap A$.
Moreover, for $1 \leq i<j \leq n$
we define $\hf_{i,j} = \frac{1}{2}(x_i f_{i,j}+f_{i,j}x_i) \in (f) \cap A$. Then
\begin{align}\label{eq.hij}
\hf_{i,j} = 
\begin{cases}
\displaystyle	y_i \prod_{\substack{(a,b) \neq (i,j)\\ a < b}} (y_a - y_b) & \text{ if } A = V_n \\
\displaystyle	(2y_i-y_j) \prod_{\substack{(a,b) \neq (i,j)\\ a < b}} (y_a - y_b) & \text{ if } A = S(1,1,-1).
\end{cases}
\end{align}

Let $J$ be the ideal generated by the 
$\binom{n}{2} + 1$ elements $\hf_{i,j}$ of degree $\binom{n}{2}$.

\begin{proposition}
\label{prop.dim}
The GK dimension of the algebra $T/J$ is at most $n - 2$.
\end{proposition}

\begin{proof}
Set $\hf = \sum_{1 \leq i < j \leq n} \hf_{i,j}$.
We will show that $\hf$ and $\VdM(T)$ are relatively prime, 
and hence they form a regular sequence. 
It follows that the grade of $J$ (that is,
the length of the longest regular sequence in $J$) is at least two,
and hence the dimension of $T/J$ is at most $n - 2$ by the depth inequality.

Let $1 \leq a < b \leq n$.  Consider the image of $\hf$ in 
$T_{a,b} = k[y_1,\dots,y_n]/( y_a - y_b )$.
Since each $\hf_{i,j}$ for $(i,j) \neq (a,b)$ has $y_a - y_b$ as a factor, the image of $\hf$
and $\hf_{a,b}$ in $T_{a,b}$ agree.  Since $T_{ab}$ is a domain and the image of all
the irreducible factors of $\hf_{a,b}$ are nonzero, the image of $\hf_{a,b}$ is nonzero as well.
Therefore $\hf$ does not have $y_a - y_b$ as a factor for any such $a,b$, 
and hence $\hf$ and $\VdM(T)$ are relatively prime.
As $T/J$ is commutative, its GK dimension is equal to its Krull dimension, so $\GKdim T/J \leq n-2$.
\end{proof}

\begin{theorem}
\label{thm.main}
The Auslander map $\gamma_{A,G} : A \# G  \to \End_{A^G}(A)$ is an isomorphism for $G$ a subgroup of $\cS_n$ acting on $A=V_n$ or $W_n$, and for $G$ a subgroup of $\cS_3$ acting on $S(1,1,-1)$.
\end{theorem}

\begin{proof}
First, let $A=V_n$ or $A=S(1,1,-1)$. We assume $n=3$ in the second case. 
Then $A$ is finitely generated over the central subalgebra $T$.
Let $A'$ be the image of the map $A \hookrightarrow A\#G \to (A\# G)/(f)$
and $T'$ the image of $T$ in $A'$.
By \cite[Lemma 5.2]{BHZ1}, $\GKdim T' = \GKdim A' = \GKdim (A\#G)/(f)$.
Clearly, $\GKdim T' \leq \GKdim T/J$ and $\GKdim T/J \leq n-2$
by the above argument. Hence
%\[ \p(A,G) = \GKdim A - \GKdim (A\#G)/(f) \geq \GKdim A - \GKdim T/J \geq n - (n-2) = 2.\]
\[ \p(A,G) = \GKdim A - \GKdim (A\#G)/(f) \geq n - (n-2) = 2.\]
Applying \cite[Theorem 0.3]{BHZ1} now completes the proof for $V_n$ and $S(1,1,-1)$.

By \cite[Proposition 3.6 and Corollary 3.7]{BHZ2}, 
$\p(W_n,G) \geq \p(V_n,G)$ and so the Auslander map is an isomorphism for $W_n$ as well.\end{proof}

\section{Pertinency computations}
\label{sec.comp}

In this section we give more precise pertinency computations
for subgroups of $\cS_3$ and $\cS_4$ acting on $V_3$ and $V_4$, respectively.
In addition, we provide techniques for computing bounds on pertinency,
distinct from the method in the previous section, including proving in Theorem \ref{thm.subgrp} that the pertinency of any subgroup is an upper bound on the pertinency of the group.  We also obtain in Theorem \ref{thm.tensor} an inequality that relates the pertinency of the actions of groups $G$ and $H$ (respectively) on algebras $A$ and $B$ (respectively) to the pertinency of $G \times H$ on the twisted tensor product $A \otimes_\tau B$ under certain conditions.

Pertinency values for $n$-cycles acting on $V_n$ were bounded,
and for $n=2^d$ computed exactly,
by Bao, He, and Zhang in a result we recall below.

\begin{theorem}[{\cite[Theorem 5.7]{BHZ1}}]
\label{thm.bhzineq}
Let $G=\grp{(1~2~\cdots~n)}$, $n \geq 2$ and let  $\phi(n)$ be Euler's phi function.
If $n=2^d$ then $\p(V_n,G)=n$.
In general,
$\p(V_n,G)\geq \frac{\phi(n)}{2}$ when $n$ is even and
$\p(V_n,G)\geq \phi(n)$ when $n$ is odd.
%\begin{equation} \label{bhz.ineq}
%\p(V_n,G)\geq
%	\begin{cases}
%		\frac{\phi(n)}{2} & $n$ \text{ even} \\
%   	\phi(n)	&	$n$ \text{ odd},
%  \end{cases}
%\end{equation}
%where $\phi(n)$ is Euler's phi function.
\end{theorem}

\begin{remark}
\label{rmk.ore}
Let $n \geq 2$ and suppose $\sigma$ is a $k$-cycle acting on $V_n$. Let $H =\grp{\sigma}$.
If $k=n$, then $\p(V_n,H)$ is as in Theorem \ref{thm.bhzineq}.
However, if $k < n$, then by reindexing we may assume $\sigma = (1~2~\cdots ~k)$.
We recognize $V_n$ as an Ore extension
of $V_k$ with $H$ acting trivially on $x_{k+1},\hdots,x_n$.
By \cite[Lemma 6.4]{BHZ2}, $\p(V_n,H)$ is still as above
but with the $n$ on the right hand side replaced by $k$.
We use this result frequently in what follows without mention.
\end{remark}

Let $G$ be a finite subgroup of $\Aut_{\gr}(A)$.
For a finite-order element $g \in G$,
the {\sf reflection number of $g$}, denoted $\r(A,g)$, is defined as
$$ \r(A,g) = \GKdim A - k, \text{ where }  \Tr_A(g,t) = \frac{1}{(1-t)^k q(t)} \text{ and } q(1)\neq 0.$$
%the difference between $\GKdim A$ and the order of the pole of $\Tr(g, t)$ at $t = 1$.
%\[\r(A,g) = \GKdim A - \text{(the order of the pole of $\Tr(g, t)$ at $t = 1$)}\]
The {\sf reflection number} of the $G$-action on $A$ is defined to be
$$\r(A,G) = \min\{ \r(A,g) : e \neq g \in G\}.$$
%\[ \r(A,G) = \min\{ \r(A,g) : e \neq g \in G\}.\]
Bao, He, and Zhang conjecture that $\p(A,G) \geq \r(A,G)$ \cite[Conjecture 0.9]{BHZ1}; if true this conjecture would imply that the Auslander map is an isomorphism for groups that contain no reflections.
For all of the permutation actions considered in this paper, we show that the conjectured inequality holds.

Next we will bound the pertinency of the group action above by the pertinency of the action of a subgroup. 
%\[f_H = \sum_{\sigma \in H} 1 \# \sigma \quad \mbox{and} \quad f_G = \sum_{\sigma \in G} 1 \# \sigma.\]

\begin{lemma}
\label{lem.subgrp}
Suppose $G$ is a group acting on an algebra $A$ and $H \leq G$ is a subgroup. Let 
$f_H = \sum_{\sigma \in H} 1 \# \sigma$ and $f_G = \sum_{\sigma \in G} 1 \# \sigma$.
If $\alpha \# e \in (f_G) \cap A \subseteq A \# G$, 
then $\alpha \# e \in (f_H) \cap A \subseteq A \# H$.
\end{lemma}

\begin{proof}
Since $\alpha \# e \in (f_G)$, 
there exist $a_i, b_i \in A$ and $\sigma_i, \sigma_i' \in G$, $1 \leq i \leq n$, such that
\[ \alpha \# e = \sum_{i =1}^n (a_i \# \sigma_i)(f_G)(b_i \# \sigma_i').\] 
Since, for all $i$, $(1\#\sigma_i) f_G = f_G = f_G (1\#\sigma_i')$, 
by replacing $b_i$ with $(\sigma_i')^{-1}(b_i)$, we may assume $\sigma_i = \sigma_i' = e$ for all $i$. Therefore,
\[ \alpha \# e = \sum_{i =1}^n (a_i \# e)\left(\sum_{\sigma \in G} 1 \# \sigma\right)(b_i \# e) 
= \sum_{i=1}^n \sum_{\sigma \in G} a_i \sigma(b_i) \# \sigma.\] 
%That is,
%\[\sum_{i=1}^n a_i \sigma(b_i) = \begin{cases} %\alpha & \mbox{if } \sigma = e \\
%0 & \mbox{otherwise.}
%\end{cases}\]
That is, $\sum_{i=1}^n a_i \sigma(b_i) = \alpha$ if $\sigma = e$ and $0$ otherwise. Now consider
\[\sum_{i =1}^n (a_i \# e)(f_H )(b_i \# e)=\sum_{i =1}^n (a_i \# e)\left(\sum_{\sigma \in H} 1 \# \sigma\right)(b_i \# e) = \sum_{i=1}^n \sum_{\sigma \in H} a_i \sigma(b_i) \# \sigma.\]
By the computation above, this is equal to $\alpha \# e$, which proves our claim.
\end{proof}

The next theorem resolves a conjecture of Bao, He, and Zhang in the
group case \cite[Remark 5.6(2)]{BHZ1}.

\begin{theorem}
\label{thm.subgrp}
Let $G$ be a group acting on $A$ and let $H \leq G$ be a subgroup. Then
\[ \p(A,G) \leq \p(A,H).\]
\end{theorem}

\begin{proof} 
By \cite[Lemma 5.2]{BHZ1},
$\GKdim A \# G/( f_G ) = \GKdim A/((f_G)\cap A)$.
Lemma \ref{lem.subgrp} implies that 
$\GKdim A/(( f_G ) \cap A) \geq \GKdim A/(( f_H ) \cap A)$
and the result follows.
%and therefore $\p(A,G) \leq \p(A,H)$.
\end{proof}

\begin{corollary}\label{cor.refl} 
Let $A$ be a noetherian connected graded $\kk$-algebra and let $G$ be a finite subgroup of $\Aut_{\gr}(A)$. 
Suppose that $g \in G$ is a reflection. If $A$ and $A^{\langle g \rangle}$ have finite global
dimension, then the Auslander map $\gamma_{A,G}$ is not an isomorphism.
\end{corollary}
\begin{proof}
Let $H = \langle g \rangle$ be the subgroup of $G$ generated by the reflection $g$. 
By \cite[Lemmas 1.10 and 1.11]{KKZ1}, the graded $A^H$-module $A$ is finitely generated and free so
$\End_{A^H}(A)$ is a matrix ring over $A^H$.  This implies that $A\#H$ is not isomorphic
to $\End_{A^H}(A)$ as $A\# H$ is concentrated in nonnegative degree and there must be a map
of negative degree in $\End_{A^H}(A)$.  One therefore has $\p(A,G) \leq \p(A,H) < 2$ and so
$\gamma_{A,G}$ is not an isomorphism.
%Dr. Moore (``Frank" amongst his homies) will complete this proof.
\end{proof}

\begin{corollary}
For $n \geq 2$, $\p(V_n, \cS_n) = 2$.
\end{corollary}
\begin{proof}
By Theorem \ref{thm.main}, $\p(V_n,\cS_n) \geq 2$.
The cyclic subgroup generated by a 2-cycle is a subgroup of $\cS_n$ with pertinency exactly equal to $2$
(Theorem \ref{thm.bhzineq}), so $\p(V_n, \cS_n) \leq 2$ by Theorem \ref{thm.subgrp}.
\end{proof}

The above theorems, combined with prior results, are summarized in the table below 
as they apply to nontrivial subgroups of $\cS_3$ acting on $V_3$.
It is clear that pertinency is stable under conjugation, so we list the subgroups up to conjugacy.

%\vspace{.5em} The submission guidelines explicitly say to avoid vertical spacing commands

%\def\arraystretch{1.5}
%\setlength{\tabcolsep}{2em}
\setlength{\tabcolsep}{.5em}
\begin{center}
\begin{tabular}{|c|c|c|c|c|} \hline
subgroup   & $\p(V_3,G)$		& reason & $\r(V_3,G)$ \\ \hline
%$\grp{e}$         & $3$		& trivial   & 	     \\ \hline
$\grp{(1~2)}$      & $2$ & Theorem \ref{thm.bhzineq}  & 2      \\ \hline
$\grp{(1~2~3)}$     & $2$ or $3$	& Theorem \ref{thm.main} & 2 \\ \hline
$\grp{(1~2),(2~3)}$ & $2$ & Theorems \ref{thm.main}, \ref{thm.subgrp} & 2 \\ \hline
\end{tabular}
\end{center}
%\vspace{.5em} The submission guidelines explicitly say to avoid vertical spacing commands

%\begin{question}
%\label{q.3cycle}
%What is $\p(V_3,\grp{(123)})$? We conjecture that the value is $2$.
%\end{question}
We need a few additional results to handle the subgroups of $\cS_4$ acting on $V_4$.

\begin{proposition} 
\label{prop.rob} 
Let $G$ be the Klein 4 subgroup $\langle (1~2)(3~4), (1~3)(2~4)\rangle$ acting on $A=V_4$. Then $\p(A,G)=4$.
\end{proposition}

\begin{proof}
Let $f = \sum_{\sigma \in G} 1 \# \sigma$.
%We will show that for $1 \leq i \leq 4$, $x_i^4 \# e \in (f)$, so $\dim_\kk A/((f) \cap A) < \infty$ and
%hence $\GKdim A\#G/(f) = \GKdim A/((f) \cap A) = 0$.
Observe that
\begin{align*}q_{(1~2)(3~4)} &:=\frac{1}{2}(x_1 + x_2 - x_3 - x_4)\left[(x_1 + x_2 - x_3 - x_4) f + f (x_1 + x_2 - x_3 - x_4) \right] \\ &= (x_1^2 + x_2^2 + x_3^2 + x_4^2) \# e + (x_1^2 + x_2^2 + x_3^2 + x_4^2) \# (1~ 2)(3~ 4)  \in (f). 
\end{align*}
Similarly, for each $\sigma \in G$, $q_{\sigma} = (x_1^2 + x_2^2 + x_3^2 + x_4^2) \# e + (x_1^2 + x_2^2 + x_3^2 + x_4^2) \# \sigma  \in (f)$. Additionally, $(x_1^2 + x_2^2 + x_3^2 + x_4^2) f \in (f)$, so by subtracting $\sum_{\sigma \in G} q_{\sigma}$, we have 
\[ r = (x_1^2 + x_2^2 + x_3^2 + x_4^2) \in (f) \cap A.\]

We also have $s = (x_1^2 + x_2^2 - x_3^2 - x_4^2) f + f (x_1^2 + x_2^2 - x_3^2 - x_4^2) \in (f)$ and therefore, $(x_1 - x_2)r + x_1s-sx_2 = (x_1 - x_2) (x_1^2 + x_2^2) \in (f) \cap A$. Indeed, by the same argument, for any $1 \leq i \neq j \leq 4$, $f_{i,j}= (x_i - x_j) (x_i^2 + x_j^2) \in (f) \cap A.$ Now, we see that $x_if_{i,j} + f_{i,j}x_i = 2(x_i^4 + x_i^2x_j^2) \in (f) \cap A$ and 
$(x_i-x_j) f_{i,j} = x_i^4 + 2x_i^2x_j^2 + x_j^4 \in (f) \cap A$. 
Subtracting, it follows that $x_i^4 - x_j^4 \in (f) \cap A$.  
Therefore, we have that $x_i^2x_j^2 + x_k^4 \in (f) \cap A$ for any $1 \leq i \neq j \leq 4$, 
and $1 \leq k \leq 4$.

Now $r^2$ is congruent to $-8x_i^4$ modulo the other six relations, 
and therefore $x_i^4 \in (f) \cap A$ for all $i$, completing our proof.
\end{proof}
Using Molien's Theorem, one can show that in lowest terms, the Hilbert series of $A^G$ as in
Proposition \ref{prop.rob} is
%is $H_{A^G}(t) = (1 - 3t + 5 t^2 - 3t^3 +t^4)(1+t^2)^{-1}(1-t)^{-4}$,
%has numerator equal to $1 - 3t + 5 t^2 - 3t^3 +t^4$
$$H_{A^G}(t) = \frac{1 - 3t + 5 t^2 - 3t^3 +t^4}{(1+t^2)(1-t)^4},$$
and so $A^G$ is a rather complicated AS Gorenstein ring.  Hence, Theorem \ref{thm.isol} provides useful information for graded isolated singularities.
\begin{proposition}
\label{prop.frank}
Let $\sigma = (1\ 2)(3\ 4)\cdots(2n-1\ 2n)$.  Then $\p(V_{2n}, \langle \sigma \rangle) = 2n$.
\end{proposition}

\begin{proof}
Let $f = 1 \# e + 1 \# \sigma$.  
Then for $i$ odd, one has $x_i f - f x_{i+1} = (x_i - x_{i+1})$.
Now skew-commuting $x_i - x_{i+1}$ with either $x_i$ or $x_{i+1}$ shows that 
$x_i^2$ and $x_{i+1}^2$ are in $(f) \cap A$. 
Hence, $\dim_\kk A/ ((f) \cap A) < \infty$ and $\p(A,\langle \sigma \rangle) = 2n$.
\end{proof}
We list the pertinencies and reflection numbers of subgroups of $\cS_4$ acting on $V_4$.
\vspace{-1em}
\begin{center}
\begin{tabular}{|c|c|c|c|c|} \hline
subgroup(s)		& $\p(V_4,G)$ & reason  & $\r(V_4,G)$  \\ \hline
%$\grp{e}$ & $4$           & trivial  &         \\ \hline
$\grp{(1~2)}$  & $2$ & Theorem \ref{thm.bhzineq} & 2       \\ \hline
$\grp{(1~2)(3~4)}$          & $4$           & Proposition \ref{prop.frank}                   & 4       \\ \hline
$\grp{(1~2~3)}$, $\grp{(1~2~3),(1~2~4)}$          & $2$ or $3$ & Theorems \ref{thm.main}, \ref{thm.subgrp} & 2       \\ \hline
$\grp{(1~2~3~4)}$            & $4$           & Theorem \ref{thm.bhzineq}                  & 4       \\ \hline
$\grp{(1~2)(3~4),(1~3)(2~4)}$ & $4$           & Proposition \ref{prop.rob}                     & 4       \\ \hline
\makecell{$\grp{(1~2~3~4),(2~4)}$, $\grp{(1~2),(3~4)}$,\\ $\grp{(1~2~3~4),(1~2)}$, $\grp{(1~2~3),(1~2)}$}       & $2$           & Theorems \ref{thm.main}, \ref{thm.subgrp} & 2       \\ \hline
\end{tabular}
\end{center}

We conjecture that $\p(V_3,\grp{(123)})=2$. 
If this conjecture holds, then $\p(V_4,G)=2$ for $G=\grp{(1~2~3)}$
and $\grp{(1~2~3),(1~2~4)}$. 
We now give an alternate way to realize an upper bound on $\p(V_4, \grp{(1~2),(3~4)})$.

Let $A$ and $B$ be algebras with multiplication maps $\mu_A$ and $\mu_B$, 
respectively. 
A $\kk$-linear homomorphism $\tau:B \tensor A \rightarrow A \tensor B$ is a  
{\sf twisting map} provided $\tau(b \tensor 1_A)=1_A \tensor b$ and 
$\tau(1_B \tensor a) = a \tensor 1_B$, $a \in A$, $b \in B$.
A multiplication on $A \tensor B$ is then given
by $\mu_\tau := (\mu_A \tensor \mu_B) \circ (\id_A \tensor \tau \tensor \id_B)$.
By \cite[Proposition 2.3]{cap}, $\mu_\tau$ is associative if and only if 
$\tau \circ (\mu_B \tensor \mu_A) 
	= \mu_\tau \circ (\tau \tensor \tau) \circ (\id_B \tensor \tau \tensor \id_A)$
as maps $B \tensor B \tensor A \tensor A \rightarrow A \tensor B$.
The triple  $(A \tensor B,\mu_\tau)$ is
a {\sf twisted tensor product} of $A$ and $B$, 
denoted by $A \tensor_\tau B$.

Let $G$ and $H$ be groups acting on $A$ and $B$, respectively.
Then $G \times H$ acts naturally on $A \tensor B$ by
$(g,h)(a \tensor b) = g(a) \tensor h(b)$.

\begin{lemma}
\label{lem.tensor}
Let $G$ and $H$ be finite groups of automorphisms 
acting on algebras $A$ and $B$.
Let $\tau : B \tensor A \rightarrow A \tensor B$ be a twisting map. 
Then $G \times H$ acts on $A \tensor_\tau B$ provided
\begin{align}
\label{eq.taucond}
\tau( (h,g)(b \tensor a) ) = (g,h)\tau(a \tensor b).
\end{align}
\end{lemma}

\begin{proof}
As observed above, $G \times H$ acts naturally on $A \tensor B$
and similarly $H \times G$ acts on $B \tensor A$.
Let $\mu_A$ and $\mu_B$ be the multiplication maps on $A$ and $B$,
respectively, and $\mu_\tau$ the (twisted) multiplication on 
$A \tensor_\tau B$.
Recall that $A \tensor_\tau B$ has the same basis as $A \tensor B$. 

Let $a \tensor b, a' \tensor b' \in A \tensor _\tau B$
and $\sum a_i'' \tensor b_i'' = \tau(b \tensor a')$.
Assuming \eqref{eq.taucond},
\begin{align*}
\mu_\tau((g,h)(a \tensor b) \tensor (g,h)(a' \tensor b'))
	&= \mu_\tau((g(a) \tensor h(b)) \tensor (g(a') \tensor h(b'))) \\
%	&= (\mu_A \tensor \mu_B)(g(a) \tensor \tau(h(b) \tensor g(a')) \tensor h(b')) \\
	&= (\mu_A \tensor \mu_B)(g(a) \tensor \tau( (h,g)(b \tensor a') ) \tensor h(b')) \\
	&= (\mu_A \tensor \mu_B)(g(a) \tensor (g,h)\tau(a' \tensor b) \tensor h(b')) \\
%	&= (\mu_A \tensor \mu_B)(g(a) \tensor (g,h)\left(\sum a_i'' \tensor b_i''\right) \tensor h(b')) \\
	&= \sum g(a)g(a_i'') \tensor h(b_i'')h(b') \\
%	&= (g,h)\left( \sum aa_i'' \tensor b_i'' b'\right) \\
%	&= \mu_\tau((g,h)( a (\sum a_i'' \tensor b_i'') b)) \\
	&= (g,h)((\mu_A \tensor \mu_B)( a \tensor \tau(b \tensor a) \tensor b)) \\
	&= (g,h)(\mu_\tau((a \tensor b)\tensor(a' \tensor b') )). \qedhere
\end{align*}
\end{proof}

\begin{example}
\label{ex.ttensor1}
Suppose $A=\kk_{-1}[x_1,\hdots,x_n]$ and $B=\kk_{-1}[y_1,\hdots,y_m]$.
Define a twisting map $\tau:B \tensor A \rightarrow A \tensor B$
by $b \tensor a \mapsto (-1)^{k\ell}(a \tensor b)$
for $a \in A_k$ and $b \in B_\ell$ and extend $\tau$ linearly.
Then $A \tensor_\tau B \iso V_{n+m}$.
If $G$ and $H$ are any groups acting linearly as automorphisms
on $A$ and $B$ respectively then they preserve degree and hence 
\eqref{eq.taucond} holds in this case.
\end{example}

The proof of the next theorem should be compared to \cite[Lemma 6.4]{BHZ2}.

\begin{theorem}
\label{thm.tensor}
Let $A$ and $B$ be affine algebras generated in degree 1 
and let $\tau:B \tensor A \to A \tensor B$ be a twisting map.
Assume $A \tensor_\tau B$ is Cohen-Macaulay.
Let $G$ and $H$ be finite subgroups of automorphisms 
acting on $A$ and $B$, respectively.
If \eqref{eq.taucond} holds for $G$, $H$, and $\tau$, then one has the inequality
\begin{align}
\label{eq.ttensor}
\max(\p(A,G),\p(B,H)) \geq \p(A \tensor_\tau B, G \times H).
\end{align}
\end{theorem}

\begin{proof}
It suffices by Theorem \ref{thm.subgrp} to prove that
$\p(A \tensor B,G \times \{e_H\}) = \p(A,G)$ and
$\p(A \tensor B,\{e_G\} \times H) = \p(B,H)$.
%The result will then follow from Theorem \ref{thm.subgrp}.
We will prove the first and the second follows similarly.

Since $G \times H$ satisfies \eqref{eq.taucond},
then clearly so does its subgroup $G \times \{e_H\}$.
By a modification of \cite[Proposition 3.11]{KL},
$\GKdim(A \tensor_\tau B) = \GKdim(A) + \GKdim(B)$.

Set $F = \sum_{g \in G} 1 \# (g,e_H)$. Then $F$ commutes with $B$ and so
\[ \left( (A \tensor_\tau B)\#(G \times \{e_H\}) \right)/(F)
	\iso \left( (A\#G) \tensor_\tau B \right)/(F)
	\iso (A\#G)/(f_G) \tensor_\tau B.\]
Now,
\begin{align*}
\p(A \tensor B,G \tensor \{e_H\})
	&= \GKdim (A \tensor_\tau B) 
		- \GKdim\left( (A \tensor_\tau B)\#(G \times \{e_H\}) \right)/(F) \\
	&= (\GKdim(A) + \GKdim(B)) - \GKdim((A\#G)/(f_G) \tensor_\tau B) \\
%	&= \GKdim(A) + \GKdim(B) - \left( \GKdim(A\#G)/(f_G) + \GKdim(B) \right) \\
%	&= \GKdim(A) - \GKdim(A\#G)/(f_G) \\
	&= \p(A,G).\qedhere
\end{align*}
\end{proof}

\begin{question}
Under what hypotheses does equality hold in \eqref{eq.ttensor}?
\end{question}

\begin{example}
\label{ex.ttensor2}
Let $A=\kk_{-1}[y_1,y_2,y_3,y_4]$ and $B=\kk_{-1}[z_1,z_2,z_3,z_4]$.
Let $\tau$ be as in Example \ref{ex.ttensor1}.
Then $A \tensor_\tau B \iso V_8$.
Let $G$ be the subgroup of $\cS_4$ generated by $(1~2~3~4)$.
Then $G$ acts on both $A$ and $B$ and so $\p(A,G) = \p(B,G) = 4$
by Theorem \ref{thm.bhzineq}.
By Theorems \ref{thm.main} and \ref{thm.tensor} we have
\[ 2 \leq \p(V_8,\{(1~2~3~4),(5~6~7~8)\}) = \p(A \tensor_\tau B, G \times G) \leq \p(A,G) = 4.\]
\end{example}

\section{The down-up algebra
\texorpdfstring{$A(-2,-1)$}{A(-2,-1)}}
\label{sec.special}

%A more general version of Theorem \ref{thm.main} would apply to
%groups of {\it weighted} permutations acting on $V_n$.
In this section, we demonstrate that our methods can
be applied to certain groups of weighted permutations.
Bao, He, and Zhang proved that the Auslander map is an isomorphism for noetherian
$A(\alpha,\beta)$ and any finite group of graded automorphisms when $\beta \neq -1$,
and also for the case $A(2,-1)$ \cite[Theorem 0.6]{BHZ2}. Here we prove that for any group acting 
on the down-up algebra $A(-2,-1)$ the Auslander map is an isomorphism
(Theorem \ref{thm.special}) solving
one of the remaining open cases.

Recall that $A=A(-2,-1)$ is the $\kk$-algebra generated by 
$x$ and $y$ subject to the two cubic relations
$x^2y + yx^2 + 2xyx = xy^2 + y^2x + 2yxy = 0$.
The element $z=xy+yx$ is central in $A$.

By \cite[Proposition 1.1]{KK}, 
\[ \Aut_{\gr}(A) = \left\lbrace
\begin{bmatrix}
a_{11} & 0 \\ 0 & a_{22}
\end{bmatrix},
\begin{bmatrix}
0 & a_{12} \\ a_{21} & 0
\end{bmatrix}
: a_{11},a_{12},a_{21},a_{22} \in \kk^\times
\right\rbrace.
\]

There exists a filtration $\cF=\{F_n\}$ defined by $F_n A  = (\kk \oplus \kk x \oplus \kk y \oplus \kk z)^n \subset A$
for all $n \geq 0$. The associated graded ring $\gr_\cF(A) \iso V_3$ \cite[Lemma 7.2(2)]{KKZ3}.
Note that $\cF$ is stable with respect to the action of $\Aut_{\gr}(A)$ on $A$. 
Moreover, $V_3$ is a connected graded algebra and the group of automorphisms of $V_3$ induced by the filtration is
\[ G := \left\lbrace
\begin{bmatrix}
a & 0 & 0 \\ 0 & b & 0 \\ 0 & 0 & ab
\end{bmatrix},
\begin{bmatrix}
0 & c & 0 \\ d & 0 & 0 \\ 0 & 0 & cd
\end{bmatrix} :
a,b,c,d \in \kk^\times
\right\rbrace.\]
As a $G$-module, $A \iso V_3$ so the $G$-action is inner faithful and homogeneous.
Let $H$ be a finite subgroup of $\Aut_{\gr}(A)$.
We identify $H$ with the corresponding subgroup of $G$,
which, by an abuse of notation, we also call $H$.
By \cite[Proposition 3.6]{BHZ2},
$\p(A,H) \geq \p(V_3,H)$.
Hence, it suffices to prove $\p(V_3,H) \geq 2$.

There are a few special cases of such groups.
If $H$ is diagonal, then the Auslander map is an isomorphism
by \cite[Theorem 5.5]{BHZ2}.
In particular, $H$ is small in the commutative sense
and the action of $H$ commutes with the graded twist
sending $V_3$ to $\kk[x_1,x_2,x_3]$.
Also, if $H$ is small (in the commutative sense) when acting on 
$T=\kk[y_1,y_2,y_3] \subset C(V_3)$ where $y_i = x_i^2 \in C(V_3)$, 
then the Auslander map is an isomorphism by classical results.

Set $f = \sum_{h \in H} 1 \# h$ and write $H = H_d \cup H_t$
where the $H_d$ are diagonal and the $H_t$ are not.
Denote elements of the first type as $M(a,b)$ and elements of the
second type as $N(c,d)$.

Let $M =\diag(a,b,ab) \in H_d$.
If two of $a,b,ab$ are equal to $1$ then so is the third.
Hence, we conclude that for every nonidentity element of $H_d$,
at least two of the entries are not $1$.
%If $a \neq 1$, then $x_1f-a\inv fx_1$ kills the $M$ component.
%Similar arguments hold for $x_2,x_3$ and $b,ab$, respectively. 
Define $S_t = \{ h_{2,1} : h \in H_t \}$, 
$S_d  = \{ h_{1,1} : h \in H_d, h_{1,1}\neq 1 \}$, and
$S_d' = \{ h_{2,2} : h \in H_d, h_{1,1}= 1, h_{2,2} \neq 1 \}$.
%\begin{align*}
%	S_d  &= \{ h_{1,1} : h \in H_d, h_{1,1}\neq 1 \} & & 
%   S_t = \{ h_{2,1} : h \in H_t \} \\
%   S_d' &= \{ h_{2,2} : h \in H_d, h_{1,1}= 1, h_{2,2} \neq 1 \}. & &    
%\end{align*}
%Moreover, $x_1^2f - d^{-2}fx_2^2$ kills the $N(c,d)$ component. Define
%Note that if $cd \neq 1$, then $zf-(cd)\inv fz$ also kills the $N$ component.

\begin{lemma}
\label{lem.spvdm}
The element
$\displaystyle V = x_1^{2|S_d|} x_2^{2|S_d'|} \prod_{d \in S_t}(x_1^2 - d^{-2}x_2^2) 
\in (f) \cap T$.
\end{lemma}

\begin{proof}
Set $p_0 = f$.
Enumerate the elements of $S_t$: $d_1,\hdots,d_n$.
Inductively define $p_k = x_1^2p_{k-1} - d_k^{-2}p_{k-1}x_2^2$.
Set $\widehat{V} = p_n$. 
All components corresponding to $H_t$ are now zero in $\widehat{V}$
and the coefficient of the identity component is $\prod_{d \in S_t}(x_1^2 - d^{-2}x_2^2)$.

Enumerate the elements of $S_d$: $a_1,\hdots,a_\ell$,
and the elements of $S_d'$: $b_1,\hdots,b_r$.
Set $q_0 = \widehat{V}$ and inductively define
\[ q_k = \begin{cases}
x_1(x_1q_{k-1}-a_k\inv q_{k-1}x_1) & 0 \leq k \leq \ell \\
x_2(x_2q_{k-1}-b_{k-\ell}\inv q_{k-1}x_2) & \ell+1 \leq k \leq \ell+r.
\end{cases}\]
Then $V = q_{r+\ell}$.
\end{proof}

\begin{lemma}
\label{lem.sprel}
There exists $\mu \in (f) \cap T$ such that $V$ and $\mu$ are relatively prime.
\end{lemma}

\begin{proof}
We claim that for each factor $v \in V$, there exists $\mu_v \in (f) \cap T$
such that $v \nmid \mu_v$ and $u \mid \mu_v$
for all other factors $u$ of $V$ relatively prime to $v$.  
It then follows that $V$ and
$\mu=\sum_v \mu_v$ are relatively prime.

Suppose $v=x_1^{2r}$ and let $\widehat{V}$ be as in Lemma \ref{lem.spvdm}.
The proof follows almost identically to that lemma after replacing $S_d,S_d'$ with
$\hat{S}_d  = \{ h_{2,2} : h \in H_d, h_{2,2}\neq 1 \}$ and
$\hat{S}_d' = \{ h_{3,3} : h \in H_d, h_{2,2}= 1, h_{3,3} \neq 1 \}$.
%\[
%	\hat{S}_d  = \{ h_{2,2} : h \in H_d, h_{2,2}\neq 1 \}, \quad
%   \hat{S}_d' = \{ h_{3,3} : h \in H_d, h_{2,2}= 1, h_{3,3} \neq 1 \}.
%\]
%Set $q_0 = \widehat{V}$ and inductively define
%\[ q_k = \begin{cases}
%x_1(x_1q_{k-1}-b_k\inv q_{k-1}x_1) & 0 \leq k \leq r \\
%x_2(x_2q_{k-1}-c_{k-r}\inv q_{k-1}x_2) & r+1 \leq k \leq r+s.
%\end{cases}\]
Then $\gamma_v = q_{r+s} \in (f) \cap T$.
The proof is similar when $v=x_2^{2(\ell-r)}$.

Suppose $v = (x_1^2 - d^{-2}x_2^2)$ for some $d$.
Define $H_t'$ as the set of $h \in H_t$ that are killed
by $d^{-2}$ and not by any other constant.
We may assume $H_t' \neq \emptyset$.
Set $p_0 = f$.
Enumerate the elements of $S_t$: $d=d_1,\hdots,d_n$, and assume $d=d_n$.
Inductively define $p_k = x_1^2p_{k-1} - d_k^{-2}p_{k-1}x_2^2$.
Hence, all components in $p_{n-1}$ corresponding to $H_t$ 
are zero except those in $H_t'$.

Let $\tilde{S} = \{ h_{3,3} : h \in H_d \cup H_t', h_{3,3} \neq 1\}$.
Enumerate the elements of $\tilde{S}:\beta_1,\hdots,\beta_m$.
Set $q_0=p_{n-1}$. Inductively define 
$q_k = x_3(x_3q_{k-1}-\beta_k\inv q_{k-1}x_3)$.
Hence, the only remaining nonzero components in $q_m$ 
have one of the following forms ($d$ fixed, $a$ arbitrary):
$M_a = M(a,a\inv)$, $N_1 = N(d\inv,d)$, $N_2 = N(-d\inv,-d)$.

Consider $q = x_1q_m-d\inv q_mx_2$, which kills the $N_1$ component. 
The coefficient in the $M_a$ component is $(x_1-(ad)\inv x_2)$,
and the coefficient in the $N_2$ component is $2x_1$.
Set $r_0 = x_2q-dqx_1$.
Since $(x_2)(x_1)-d(x_1)(-d\inv x_2) = x_2x_1+x_1x_2=0$,
then the $N_2$ component is zero in $r_0$.
What remains in the $M_a$ component is
$2x_2x_1 - ad( x_1^2 + (ad)^{-2} x_2^2)$.
%\[ x_2(x_1-(ad)\inv x_2)-d(x_1-(ad)\inv x_2)(ax_1)
%	= 2x_2x_1 - ad( x_1^2 + (ad)^{-2} x_2^2).\] 

Enumerate the distinct $(1,1)$-entries of the $M_a$ components
whose coefficient is not zero: $a_1,\hdots,a_n$, with $a_1=1$. 
Inductively define 
\[ 
r_k = \begin{cases}
x_1(x_1 r_{k-1} + a_k\inv r_{k-1} x_1) & k \text{ odd} \\
x_2(x_2 r_{k-1} + a_k\inv r_{k-1} x_2) & k \text{ even}.
\end{cases}
\]
All coefficients in $r_n$ are central and it remains only 
to kill the remaining nonidentity $M_a$ components.
For this we can proceed as in Lemma \ref{lem.spvdm}.
\end{proof}

\begin{theorem}
\label{thm.special}
Let $A = A(-2,-1)$. The Auslander map $\gamma_{A,G}:A\#G \rightarrow \End_{A^G}(A)$ is
a graded isomorphism for any finite subgroup $G$ of $\Aut_{\gr}(A)$.
\end{theorem}

\begin{proof}
As noted previously, it suffices by \cite[Proposition 3.6]{BHZ2} 
to prove this for $V_3$ and the corresponding group $H$ acting on $V_3$.
By Lemmas \ref{lem.spvdm} and \ref{lem.sprel},
$(f) \cap T$ contains two relatively prime elements
(see Proposition \ref{prop.dim}).  We may now
proceed as in Theorem \ref{thm.main}.
\end{proof}

We observe that by \cite[Proposition 6.4]{KKZ1}, 
the group $G$ in Theorem \ref{thm.special} contains no reflections.

\section{Graded isolated singularities}
\label{sec.sing}

%Recall that in the setting of this paper, $A^G$ is a graded isolated singularity if and only if $A\#G/(f_G)$ is
%finite-dimensional.
Mori and Ueyama proved that if the Auslander map is an isomorphism, then $A^G$ is a graded isolated 
singularity if and only if $A\#G/(f_G)$ is finite-dimensional \cite[Theorem 3.10]{MU}. Thus, we have shown 
that $V_4^G$ is a graded isolated singularity when
$G=\grp{(1~2)(3~4)}$ (Proposition \ref{prop.frank}) or $\grp{(1~2)(3~4), (1~3)(2~4)}$
(Proposition \ref{prop.rob}).  In this section we present 
additional examples of graded isolated singularities using our methods. For these examples, we
therefore obtain the conclusions of Theorem~\ref{thm.isol} as a corollary.

\subsection{Generic three-dimensional Sklyanin algebras}

Let $A=S(a,b,c)$ for a generic choice of $a,b,c \in \kk$.
Then $C_3 = \langle \sigma\rangle $ acts on $A$ with $\sigma = (1~2~3)$
permuting the variables.  We may diagonalize this action by choosing a different generating 
set of the algebra.  Indeed, let $\xi$ be a primitive third root of unity, and set
$X = x_1 + \xi x_2 + \xi^2 x_3$, $Y = x_1 + \xi^2 x_2 + \xi x_3$ and $Z = x_1 + x_2 + x_3$.  
Then $A$ is also generated by $X,Y$ and $Z$.  To describe the relations that $X,Y$
and $Z$ satisfy, we have the following lemma:

\begin{lemma}
In the notation of the previous paragraph, the generators $X,Y,Z$ satisfy Sklyanin
relations with parameters
$$(\alpha,\beta,\gamma) := (c + \xi a + \xi^2 b, c + \xi^2 a + \xi b, c + a + b)$$
\end{lemma}

\begin{proof}
Let $f_1 = axy + byx + cz^2$, $f_2 = ayz + bzy + cx^2$, and $f_3 = azx + bxz + cy^2$.
Define $F_1,F_2$ and $F_3$ similarly in terms of the generators $X,Y$ and $Z$ and the 
parameters $\alpha,\beta$ and $\gamma$.  A direct verification shows that one has the 
following equalities which complete the proof of the lemma:
\begin{align*}
F_1 & = \frac{1}{3}(f_1 + f_2 + f_3) \\
F_2 & = \frac{1}{3}(\xi f_1 + f_2 + \xi^2f_3) \\
F_3 & = \frac{1}{3}(\xi^2f_1 + f_2 + \xi f_3). \qedhere
\end{align*}
%\alpha XY + \beta YX + \gamma Z^2 & = & \frac{1}{3}((axy+byx+cz^2) + (ayz + bzy + cx^2) + (azx + bxz + cy^2)) \\
%\alpha YZ + \beta ZY + \gamma X^2 & = & \frac{1}{3}(\xi(axy+byx+cz^2) + (ayz + bzy + cx^2) + \xi^2(azx + bxz + cy^2)) \\
%\alpha ZX + \beta XZ + \gamma Y^2 & = & \frac{1}{3}(\xi^2(axy+byx+cz^2) + (ayz + bzy + cx^2) + \xi(azx + bxz + cy^2)). \\
\end{proof}

\begin{theorem}
\label{thm.sklyanin}
Let $A$ be a Sklyanin algebra on generators $x,y,z$ with parameters $(a : b : c)$
and let $(\alpha : \beta : \gamma)$ be the coordinates of the Sklyanin algebra
on generators $X,Y,Z$ as in the discussion above.  If $\alpha,\beta,\gamma$, and
$\alpha^3 - \beta^3$ are nonzero and $\sigma$ is the permutation action of
$(x~y~z)$ on $A$, then $A^{\grp{\sigma}}$ is a graded isolated singularity.
\end{theorem}

\begin{proof}
It is easy to check that in our new coordinates, the action of $\sigma$ is given by
$\sigma(X) = \xi^2 X$, $\sigma(Y) = \xi (Y)$ and $\sigma(Z) = Z$.
First, we claim that $Y^2 \# e$ is in the ideal generated by $f = f_G$.
Indeed, one has that
\[Yf - \xi fY = (1 - \xi)Y \# e + (1 - \xi^2)Y \# \sigma,\]
hence $g = Y \# e + (1 - \xi)Y \# \sigma$ is in $(f)$.  
The computation $Yg - \xi^2 gY = (1 - \xi^2)Y^2 \# e$ proves the desired claim.
A similar calculation shows that $X^2 \# e$ is in $(f)$.

Next, let $I$ be the ideal generated by the Sklyanin relations on
$X,Y$ and $Z$, and $X^2,Y^2$.
Using the generators $X^2$ and $Y^2$ we may simplify these relations to:
$$\alpha XY + \beta YX + \gamma Z^2,\alpha YZ + \beta ZY, X^2, \alpha ZX + \beta XZ, Y^2.$$
We use the diamond lemma with $X > Y > Z$ to produce a (partial) Gr\"{o}bner basis of $I$.
In the computations, we will assume that $\alpha,\beta,\gamma$, and 
$\alpha^3-\beta^3$ are nonzero

First, we compute the overlap between $\alpha XY + \beta YX + \gamma Z^2$ and $X^2$. 
In this calculation we use $=$ to denote equality in the
tensor algebra and $\leadsto$ to indicate a reduction has taken place
using the generators of the Gr\"obner basis of the defining ideal that have
been found thus far. We have
\begin{eqnarray*}
X\left(\alpha XY + \beta YX + \gamma Z^2\right) - X^2(\alpha Y) & = & \beta XYX + \gamma XZ^2 \leadsto \\
 \gamma XZ^2 - \frac{\beta^2}{\alpha}YX^2 - \frac{\beta\gamma}{\alpha}Z^2X
   & \leadsto & -\frac{\alpha\gamma}{\beta}ZXZ - \frac{\beta\gamma}{\alpha}Z^2X \leadsto \\
 \left(\frac{\alpha^2\gamma}{\beta^2} - \frac{\beta\gamma}{\alpha}\right)Z^2X
   & = & \left(\frac{\gamma(\alpha^3 - \beta^3)}{\alpha\beta^2}\right)Z^2X.
\end{eqnarray*}
Therefore $Z^2X$ is in the ideal under our assumptions on the parameters.
A similar calculation shows that $Z^2Y$ is in the ideal.
%under the same conditions.
Computing the overlap between $Z^2X$ and $\alpha XY + \beta YX + \gamma Z^2$ shows
that $Z^4$ is in the ideal as well.

%Next, we compute the overlap between $Z^2X$ and $\alpha XY + \beta YX + \gamma Z^2$:
%\begin{equation*}
%(Z^2X)(\alpha Y) - (Z^2)(\alpha XY + \beta YX + \gamma Z^2) =  -\beta Z^2YX - \gamma Z^4 \leadsto \gamma Z^4.
%\end{equation*}
%This shows that as long as $\gamma \neq 0$, $Z^4$ is in the ideal.

Hence, under the hypotheses on the parameters,
the ideal $I$ contains elements with lead terms $X^2$, $XY$, $Y^2$, $YZ$, $XZ$, $Z^2X$, $Z^2Y$, $Z^4$.
Therefore a spanning set of the quotient by the ideal $I$ consists of 
the monomials that do not contain those in the list above as subwords.  
A straightforward calculation shows that such a list of monomials
is given by $1$, $X$, $Y$, $Z$, $YX$, $ZX$, $ZY$, $Z^2$, $ZYX$, $Z^3$.
\end{proof}

\subsection{Weighted action on \texorpdfstring{$V_4$}{V4}}

Consider the following action of the Klein-4 group on $V_4$. 
Let $K = \grp{\alpha,\beta} \subseteq \Aut_{\gr}(A)$ where
$\alpha(x_i)=-x_i$ for all $i$ and $\beta = (1~3)(2~4)$.

\begin{theorem}
\label{thm.weighted}
Let $K$ be the group above. 
Then $\p(V_4,K)=4$ and thus $V_4^K$ is a graded isolated singularity,
as is $V_4^H$ for any subgroup $H$ of $K$.
\end{theorem}

\begin{proof}
Let $f = 1\#e + 1\#\alpha + 1\#\beta + 1\#\alpha\beta$ and set
\begin{align*}
p_1 &= (x_1+x_3)f + f(x_1+x_3) = 2(x_1+x_3)\#e + 2(x_1+x_3)\#\beta \\
p_2 &= x_1f-fx_3 = 2(x_1-x_3)\#e + 2(x_1-x_3)\#\alpha + 2(x_1-x_3)\#\alpha\beta.
\end{align*}
Then 
\begin{align*}
	x_2p_1 + p_1x_4 &= 2(x_2-x_4)(x_1+x_3) \# e \quad\text{and} \\
	(x_2+x_4)p_2-p_2(x_2+x_4) &= 2(x_2+x_4)(x_1-x_3) \#e.
\end{align*}
Combining these gives $x_2x_3-x_4x_1, x_2x_1-x_4x_3 \in (f) \cap V_4$.

Next we produce degree three elements in $(f) \cap V_4$. Set
\[ p_3 = (x_1x_3)f  + f(x_1x_3) = 2x_1x_3\# e + (2x_1x_3)\#\alpha.\]
Thus, $x_1p_3 - p_3x_1 = 4x_1^2x_3\#e \in (f) \cap V_4$.
Conjugation and symmetry now give $x_1x_3^2, x_2^2x_4,x_2x_4^2 \in (f) \cap V_4$.
Finally we set
\[ p_4 = x_1^2f  + fx_3^2 = (x_1^2-x_3^2)\# e + (x_1^2-x_3^2)\#\alpha.\]
Then $x_1p_4 + p_4x_1 = 2x_1(x_1^2-x_3^2)\#e$ and
it follows that $x_1^3 \in (f) \cap V_4$.
Similarly, $x_2^3,x_3^3,x_4^3 \in (f) \cap V_4$ hence 
$V_4/((f) \cap V_4)$ is finite-dimensional.
The remaining claim now follows from Theorem \ref{thm.subgrp}.
\end{proof}

%\section{Open questions, further research}

\subsection*{Acknowledgments}  We thank the referee for many helpful suggestions.
Many computations for this paper were performed with the computer algebra system
Macaulay2 \cite{M2}. The authors would also like to thank Andrew Conner
and James Zhang for many stimulating conversations related to this project.
E. Kirkman was partially supported by grant \#208314 from the Simons Foundation.

%\bibliography{biblio}{}

\begin{thebibliography}{10}

\bibitem{A}
M.~Auslander, \emph{On the purity of the branch locus}, Amer. J. Math.
  \textbf{84} (1962), 116--125. \MR{0137733}

\bibitem{BHZ2}
Y.-H. Bao, J.-W. He, and J.~J. Zhang, \emph{Noncommutative {A}uslander
  theorem}, Trans. Amer. Math. Soc., electronically published on June 26, 2018,
  DOI: https://doi.org/10.1090/tran/7332 (to appear in print).

\bibitem{BHZ1}
\bysame, \emph{Pertinency of {H}opf actions and quotient categories of
  {C}ohen-{M}acaulay algebras}, arXiv:1603.02346 (2016), (accepted for
  publication in J. Noncommut. Geom.).

\bibitem{BL}
K.~A. Brown and M.~Lorenz, \emph{Grothendieck groups of invariant rings and of
  group rings}, J. Algebra \textbf{166} (1994), no.~3, 423--454. \MR{1280586}

\bibitem{Buch}
R.-O. Buchweitz, \emph{Morita contexts, idempotents, and {H}ochschild
  cohomology---with applications to invariant rings}, Commutative algebra
  ({G}renoble/{L}yon, 2001), Contemp. Math., vol. 331, Amer. Math. Soc.,
  Providence, RI, 2003, pp.~25--53. \MR{2011764}

\bibitem{cap}
A.~Cap, H.~Schichl, and J.~Van{\v{z}}ura, \emph{On twisted tensor products of
  algebras}, Comm. Algebra \textbf{23} (1995), no.~12, 4701--4735. \MR{1352565
  (96k:16039)}

\bibitem{CKWZ2}
K.~Chan, E.~Kirkman, C.~Walton, and J.~Zhang, \emph{Mc{K}ay correspondence for
  semisimple {H}opf actions on regular graded algebras, {II}}, arXiv:1610.01220
  (2016), (accepted for publication in J. Noncommut. Geom.).

\bibitem{CKWZ1}
K.~Chan, E.~Kirkman, C.~Walton, and J.~J. Zhang, \emph{Mc{K}ay correspondence
  for semisimple {H}opf actions on regular graded algebras, {I}}, J. Algebra
  \textbf{508} (2018), 512--538. \MR{3810305}

\bibitem{Chev}
C.~Chevalley, \emph{Invariants of finite groups generated by reflections},
  Amer. J. Math. \textbf{77} (1955), 778--782. \MR{0072877}

\bibitem{EW}
P.~Etingof and C.~Walton, \emph{Semisimple {H}opf actions on commutative
  domains}, Adv. Math. \textbf{251} (2014), 47--61. \MR{3130334}

\bibitem{M2}
D.~R. Grayson and M.~E. Stillman, \emph{Macaulay2, a software system for
  research in algebraic geometry}, Available at
  \url{http://www.math.uiuc.edu/Macaulay2/}.

\bibitem{JZ}
P.~J{\o}rgensen and J.~J. Zhang, \emph{Gourmet's guide to {G}orensteinness},
  Adv. Math. \textbf{151} (2000), no.~2, 313--345. \MR{1758250}

\bibitem{KK}
E.~Kirkman and J.~Kuzmanovich, \emph{Fixed subrings of {N}oetherian graded
  regular rings}, J. Algebra \textbf{288} (2005), no.~2, 463--484. \MR{2146140}

\bibitem{KKZ1}
E.~Kirkman, J.~Kuzmanovich, and J.~J. Zhang, \emph{Rigidity of graded regular
  algebras}, Trans. Amer. Math. Soc. \textbf{360} (2008), no.~12, 6331--6369.
  \MR{2434290}

\bibitem{KKZ4}
\bysame, \emph{Shephard-{T}odd-{C}hevalley theorem for skew polynomial rings},
  Algebr. Represent. Theory \textbf{13} (2010), no.~2, 127--158. \MR{2601538}

\bibitem{KKZ2}
\bysame, \emph{Invariants of {$(-1)$}-skew polynomial rings under permutation
  representations}, Recent advances in representation theory, quantum groups,
  algebraic geometry, and related topics, Contemp. Math., vol. 623, Amer. Math.
  Soc., Providence, RI, 2014, pp.~155--192. \MR{3288627}

\bibitem{KKZ3}
\bysame, \emph{Invariant theory of finite group actions on down-up algebras},
  Transform. Groups \textbf{20} (2015), no.~1, 113--165. \MR{3317798}

\bibitem{KMP}
E.~Kirkman, I.~M. Musson, and D.~S. Passman, \emph{Noetherian down-up
  algebras}, Proc. Amer. Math. Soc. \textbf{127} (1999), no.~11, 3161--3167.
  \MR{1610796}

\bibitem{KL}
G\"unter~R. Krause and Thomas~H. Lenagan, \emph{Growth of algebras and
  {G}elfand-{K}irillov dimension}, revised ed., Graduate Studies in
  Mathematics, vol.~22, American Mathematical Society, Providence, RI, 2000.
  \MR{1721834}

\bibitem{L}
T.~Levasseur, \emph{Some properties of non-commutative regular graded rings},
  Glasgow Mathematical Journal \textbf{34} (1992), no.~3, 277–300.

\bibitem{M}
I.~Mori, \emph{Mc{K}ay-type correspondence for {AS}-regular algebras}, J. Lond.
  Math. Soc. (2) \textbf{88} (2013), no.~1, 97--117. \MR{3092260}

\bibitem{MU}
I.~Mori and K.~Ueyama, \emph{Ample group action on {AS}-regular algebras and
  noncommutative graded isolated singularities}, Trans. Amer. Math. Soc.
  \textbf{368} (2016), no.~10, 7359--7383. \MR{3471094}

\bibitem{SheTod}
G.~C. Shephard and J.~A. Todd, \emph{Finite unitary reflection groups},
  Canadian J. Math. \textbf{6} (1954), 274--304. \MR{0059914}

\bibitem{U}
K.~Ueyama, \emph{Graded maximal {C}ohen-{M}acaulay modules over noncommutative
  graded {G}orenstein isolated singularities}, J. Algebra \textbf{383} (2013),
  85--103. \MR{3037969}

\bibitem{U2}
\bysame, \emph{Cluster tilting modules and noncommutative projective schemes},
  Pacific J. Math. \textbf{289} (2017), 449--468. \MR{3667180}

\bibitem{WWY}
C.~Walton, X.~Wang, and M.~Yakimov, \emph{The {P}oisson geometry of the
  3-dimensional {S}klyanin algebras}, arXiv:1704.04975 (2017).

\bibitem{Z2}
J.~J. Zhang, \emph{Connected graded {G}orenstein algebras with enough normal
  elements}, Journal of Algebra \textbf{189} (1997), no.~2, 390 -- 405.

\end{thebibliography}
%\bibliographystyle{amsplain}

\providecommand{\bysame}{\leavevmode\hbox to3em{\hrulefill}\thinspace}
\providecommand{\MR}{\relax\ifhmode\unskip\space\fi MR }
\providecommand{\MRhref}[2]{%
  \href{http://www.ams.org/mathscinet-getitem?mr=#1}{#2}
}
\providecommand{\href}[2]{#2}

\end{document}